\documentclass[english]{article}
\usepackage[T1]{fontenc}
\usepackage{amsmath, amsthm,  amsfonts, amssymb, graphicx,  color, amscd, mathrsfs, enumitem, mathtools, physics, csquotes, dsfont}

\usepackage[colorlinks=true,linkcolor=black,anchorcolor=black,citecolor=black,filecolor=black,menucolor=black,runcolor=black,urlcolor=black]{hyperref}
\usepackage{comment}
\usepackage{marginnote}

\usepackage{calrsfs}

\definecolor{blue}{rgb}{0.0, 0.0, 1.0}
\definecolor{airforceblue}{rgb}{0.36, 0.54, 0.66}
\definecolor{azure(colorwheel)}{rgb}{0.0, 0.5, 1.0}
\definecolor{navyblue}{rgb}{0.0, 0.0, 0.5}
\definecolor{aoe}{rgb}{0.0, 0.5, 0.0}
\definecolor{beaublue}{rgb}{0.74, 0.83, 0.9}

\definecolor{brown}{rgb}{0.59, 0.29, 0.0}
\definecolor{teal}{rgb}{0.0, 0.5, 0.5}

\newtheorem{theorem}{Theorem}
\newtheorem{lemma}{Lemma}
\newtheorem{example}{Example}

\usepackage{authblk}
\usepackage{blindtext}
 
\begin{document}
	\title{Coercive Inequalities on Carnot Groups: Taming Singularities}
	
	\author[1]{E. Bou Dagher\thanks{Esther Bou Dagher: esther.bou-dagher17@imperial.ac.uk}}
\author[2]{B. Zegarli\'{n}ski\thanks{Bogus\l{}aw Zegarli\'{n}ski: boguslaw.zegarlinski@cnrs.fr}}
\affil[1]{Department of Mathematics, Imperial College London}
\affil[2]{Institut de Mathématiques de Toulouse ; UMR5219}

	\maketitle
	\begin{abstract}
	
	In the setting of Carnot groups, we propose an approach of taming singularities to get coercive inequalities. To this end, we develop a technique to introduce natural singularities in the energy function $U$ in order to force one of the coercivity conditions. In particular, we explore explicit constructions of probability measures on Carnot groups which secure Poincar{\' e} and even Logarithmic Sobolev inequalities. 
	As applications, we get analogues of the Dyson-Ornstein-Uhlenbeck model on the Heisenberg group and obtain results on the discreteness of the spectrum of related Markov generators.

		\tableofcontents{}
	\end{abstract}

	\section{Introduction}
	
	In the literature, there exists many models in the Euclidean space in which one may have a system with a one-particle potential perturbed by a singular multi-particle interaction. Such systems exhibit more intricate behaviour and are technically much more complicated. One famous example is provided by quantum harmonic oscillators coupled by Coulomb potential, see e.g. \cite{Du}, \cite{Ro}, \cite{RoV} and therein for an extensive earlier literature, and \cite{V} for a more recent study of Poincaré and Log-Sobolev inequalities in the related setup.
	Another interesting class of models involving singular potentials are related to random matrix theory, see e.g. \cite{D}, \cite{BCL} and references therein. In both cases, one also studies Markov generators and related stochastic processes. In the indicated models in the Euclidean space, the non-singular part of the potential provides a Markov generator with good behaviour with unique invariant measure, discrete spectrum, and satisfying various coercive inequalities (of Poincaré and Log-Sobolev type).
	For example, for a quantum harmonic oscillator, the ground state is provided by a Gaussian measure, and the related hamiltonian can be represented as the Ornstein-Uhlenbeck Markov generator, for which all the above-mentioned good properties are known. The addition of singular potential does not change the situation. When one moves to a space provided by nilpotent Lie groups, there exists many topologically equivalent metrics which could be naturally used a basis for a quadratic potential. For example, in the simplest case of the Heisenberg group, one could consider a harmonic oscillator potential given by the control distance (defined by the eikonal equation associated with the sub-gradient) or the Kaplan norm (which shows up when one considers the Green function of the sub-Laplacian). 
	It appears that the corresponding analogues of Gaussian measures in both cases are dramatically different. In the first case, we have similar coercive inequalities as in the Euclidean case. In the second, the Log-Sobolev inequality fails for certain fundamental reasons (see  \cite{HZ}) although Poincaré inequality still remains true.
	One of the things which  is demonstrated in our work is that given as a starting point the bad case (with the potential quadratic in Kaplan norm), by adding a suitable singularity, we obtain an analogue of the Dyson-Ornstein-Uhlenbeck type operator for which strong properties are restored.

	\subsection{Coercive Inequalities Problem on Nilpotent Lie Groups} 
	\label{Sec.1}
	
\noindent Let $\mathbb{G}$ be a Carnot group 
with a given $n\in\mathbb{N}$ dependent on the group, and let $\nabla\equiv (X_i)_{i=1,..,n}$, and $\Delta \equiv \sum_{i=1,..,n} X_i^2 $
denote the corresponding sub-gradient and sub-Laplacian, respectively.

In this paper, we are interested in the coercive bounds of the following form
\[ \mu\left( |f|^q \mathcal{V}\right)\leq 
C\mu |\nabla f|^q +D \mu|f|^q,
\tag{*}
\]
with some $q\in[1,\infty)$, constants $C,D\in(0,\infty),$ and some function $\mathcal{V}$ independent of a function $f$. \\
 When $\mathcal{V}(x)$ diverges to infinity
with a distance $d(x,0)\to \infty$, such bounds have a multitude of applications to prove coercive inequalities, see e.g. \cite{R}, \cite{A}, \cite{HZ},
\cite{I}, \cite{IKZ}, \cite{WZ}, \cite{ChFZ}.
In particular, what  we are after is an interesting and very challenging problem of how to get coercive inequalities for probability measures on Carnot groups. Besides the direct interest in analysis on groups, additional motivation comes from the fact that such coercive inequalities play an important role in studying ergodicity properties of Markov semigroups, spectral theory, isoperimetry, etc..., in finite and infinite dimensional spaces
(see e.g. \cite{BoZ1}, \cite{BoZ2}, \cite{LZ}, \cite{KOZ}, \cite{RZ},\cite{RZ2} and references therein).

Some possible natural approaches to the coercive bounds (*) described above are as follows. 
First, as described in \cite{HZ}, one uses Leibniz rule together with integration by parts. For a probability measure of the form $d\mu=e^{-U} d\lambda$, we have
\[
(\nabla f)e^{-\frac12 U} = (\nabla f e^{-\frac12 U})+
\frac12 f (\nabla U) e^{-\frac12 U}.
\]
Hence, squaring and integrating both sides with respect to $\lambda$, we get
\[ \begin{split}
 &\int |\nabla f|^2 e^{-U}d\lambda = \\
& \int f^2  \frac14 |\nabla U|^2e^{-U} d\lambda +\int (\nabla f e^{-\frac12 U})\cdot (\nabla  U )f e^{-\frac12 U} d\lambda  + \int (\nabla f e^{-\frac12 U})^2 d\lambda, 
\end{split}\]
and after integration by parts in the middle term and dropping out the third term on the right-hand side, we arrive at  
\[ 
\int f^2 \left(\frac14 |\nabla U|^2 -\frac12 \Delta U\right) d\mu \leq \int |\nabla f|^2 d\mu.
\tag{$\diamond$}
\]
If one can show 
that the following potential
\[\mathcal{V}_2 \equiv \frac14 |\nabla U|^2 -\frac12 \Delta U \]
is locally bounded and diverges to infinity in all directions, then by
\cite{HZ}, 
one gets  Poincar{\'e} inequality for $\mu$ with the sub-gradient $\nabla$. 
In fact, assuming Poincar{\'e} in the balls for the Lebesgue measure $\lambda$ (\cite{J}) and local boundedness of $U$, in order to obtain the Poincar{\'e} inequality with $\mu$ and the sub-gradient, it is sufficient to show the following bound 
\[ 
\int f^2 \eta d\mu \leq  A \int |\nabla f|^2 d\mu
+B \int  f^2 d\mu,
\tag{$\diamond\diamond$}
\]
with a function $\eta$ diverging to infinity with the distance from the origin, and some constants $A,B\in(0,\infty)$ independent of $f$ for which integrals on the right-hand side are well-defined, \cite{HZ}.\\

If additionally one has
\[ |\nabla U|^2 + U \leq a (1+ \mathcal{V}_2), \]
then one gets Log-Sobolev type inequalities,
\cite{HZ}, (and possibly more non-tight coercive inequalities, \cite{R}, \cite{A},  which recently were tightened in \cite{WZ}). \\

Another point of view is to start from quasi-invariant measures and corresponding unitary groups,
{\'a} la \cite{AHK}, as follows.
Assuming the measure $\mu(\cdot\circ g)$, shifted by an action of the group element $g$, is absolutely continuous with respect to $\mu$, we define one parameter unitary groups
\[
S_t f(\omega)\equiv f(\omega\circ g_t) 
 \left( \frac{d\mu(\omega\circ g_t)}{d\mu(\omega)} \right)^\frac12.
\]
These groups have generators of the form
\[
\mathbb{X}f \equiv \nabla f -f\frac12 \nabla U,
\]
which satisfy the formula of integration by parts with the measure $\mu$. Hence, we can consider a Dirichlet operator
\[
\int f (-\mathcal{L} f) d\mu \equiv \int|\mathbb{X}f |^2 d\mu \geq 0, \tag{$\star\star$}
\]
defined on a dense domain.
A simple computation shows that
\[
\mathcal{L} f = \mathbb{X}^2 f = \Delta f - \nabla U\cdot \nabla f
+\left(\frac14 |\nabla U|^2- \frac12 \Delta U \right)f  \equiv L f + \mathcal{V}_2 f.
\]
Since the operator $L$ on the right-hand side is the Dirichlet operator associated with the sub-gradient
\[
\int f (-L f)d\mu = \int |\nabla f|^2 d\mu ,
\]
the relation ($\star\star$) implies the coercive bound ($\diamond$). \\

Finally, if one is interested in the following inequality on $\mathbb{G}$ (outside some ball)
	
	\[  - \frac12\Delta U +\frac14 |\nabla U|^2  \geq \eta(N) - B ,\]
	with $\eta(x)\underset{d(x,0)\to\infty}{\longrightarrow}\infty$ and a constant $B\in(0,\infty)$,   
	one should notice that
	\[ \Delta e^{-\frac12U}=\frac12 \nabla\cdot\left(-(\nabla U)e^{-\frac12U}\right)=
	\left(-\frac12 \Delta U +\frac14 |\nabla U|^2\right) e^{-\frac12 U}.
	\]
	Hence, our problem is equivalent to the following inequality
	\[
	\Delta e^{-\frac12 U} \geq (\eta -B) e^{-\frac12U},
	\]	
	and one could consider it as the following differential inequality
	\[ \label{bigstar}
	\begin{split}
	 \exists  \Phi>0 \qquad & \Delta   \Phi 
	 \geq    (\eta-B)  \Phi.   \\ & \Phi\underset{s\to\infty}{\longrightarrow}0
	\end{split}
\tag{{$\bigstar$}}
	\]
	In more general form, for example to obtain Log-Sobolev inequalities, we may also like to have 
	\[\label{2bigstars}
	\Delta \Phi
	\geq \eta\;  \Phi \log(\frac1{\Phi}). 
	\tag{{$\bigstar\bigstar$}}
	\]
Focusing for a moment on (\ref{bigstar}), we rewrite it as follows, for $\omega\in\Omega^c$, for some ball $\Omega\subset\mathbb{G}$,
\[ 
(-\Delta +\gamma)\Phi \leq - \tilde\eta \Phi,
\]
which can be represented as
\[
0\leq \Phi \leq \Psi_0 + \int_0^\infty dt e^{-\gamma t} e^{t\Delta} (\tilde \eta \Phi),
\tag{$\bigstar.1$}\]
with some $(-\Delta +\gamma)\Psi_0(\omega)=0$ for $\omega\in\Omega^c$, $\Psi_0(\omega)\underset{d(\omega,0)\to\infty}{\longrightarrow}0$ . At this point, we recall that the  heat kernel $p_t(w,v)$ on $\mathbb{G}$ has the following estimate, e.g. \cite{H},
\[
p_t(w,v)\leq C\; t^{-Q/2}\; e^{-\frac{d^2(w,v)}{2\alpha t} },
\]
with some positive constants $C,\alpha\in(0,\infty),$ and $Q$ denoting homogeneous dimension of the group. 

Hence, ({$\bigstar.1$}) gets the following form
\[
0\leq \Phi \leq \Psi_0 + \int_0^\infty dt e^{-\gamma t} C\; t^{-Q/2}\; e^{-\frac{d^2(w,v)}{2\alpha t} } (\tilde \eta \Phi).
\tag{$\bigstar.2$}\]
	
%
\vspace{0.25cm}	

Some problems of particular interest involve 
$\Phi,$ which is a function of a homogeneous norm on the Carnot group. In particular, choosing dependence on the control distance, we can ask for what Carnot group we can find a solution to the following problem
\[
-\Phi'(d) \Delta d
+ \Phi''(d) |\nabla d|^2
\geq \eta(d)\; \Phi (d).
\] 
Taking into the account that for the control distance by definition $|\nabla d|^2=1$, we obtain the following relation 
\[
 \Delta d
\leq \frac{\Phi'' (d) }{\Phi'(d)} - \eta(d)\; \frac{\Phi (d) }{\Phi'(d)},    \tag{$\star\star\star\star$}
\]
for points where $\Phi'(d)\neq 0$; e.g. in a special but already interesting case, one could ask that for a monotone $\Phi$ with maximum only at zero. 	\\
We notice that the Laplacian distance bounds are well-known and play an important role in certain problems in Riemannian geometry, \cite{D}, \cite{ChBPL}, 
but up to our knowledge, there is no literature on this problem in the context of analysis on nilpotent Lie groups. 

We remark that one could consider a similar problem with a function $\Phi$ of some other homogeneous norm on the group. We know that, if we switch to smooth homogeneous norms, we generally cannot get the Log-Sobolev inequality, \cite{HZ}. However, there are examples in the literature of Poincar{\'e} and weaker log$^\beta$ inequalities in such setup, see \cite{I}, \cite{IKZ}  for the Heisenberg group, \cite{BDZ1}, \cite{BDZ2}, for type-2 groups, and \cite{ChFZ} for filiform-type groups. Later in this paper, we provide some general constructions which work for more extensive classes of Carnot groups.
Thus, we are led to considering the following problem
\[
\Delta \Phi(K)
\geq \eta(K)\;  \Phi(K),
\]
involving some other homogeneous norms $K.$ Using equivalence of norms, we choose to have a function $\eta \equiv \eta(K)$.
In particular, for the Kaplan type norm $N$, we have
the following general relation, \cite{BLU}
\[
	\triangle N=
	\frac{(Q-1)|\nabla N|^{2}}{N},\label{eq:1.2}
\]
where $Q$ is the homogeneous dimension of the group.
Taking this into account, our problem looks as follows
 \[
 \left( \Phi''(N)  +   \Phi'(N)\frac{(Q-1)}{N}  
 \right) |\nabla N|^{2}
 \geq \eta(N)
  \Phi(N).
\] 
Unfortunately, for smooth homogeneous norms, $\nabla N$ vanishes in some directions, so we cannot have $\eta$ growing to infinity with the distance in all directions, and the pointwise strategy doesn't work.
However, in some cases, one can overcome this problem by other arguments and get Poincar{\'e} inequality (but not Log-Sobolev), \cite{I}, \cite{BDZ1},  \cite{BDZ2}, \cite{ChFZ}.
\\

In the situation where very little is known about coercive inequalities for probability measures on Carnot groups, we explore yet another direction in this paper. We propose an approach of taming the singularities, that is, we develop a technique to introduce natural singularities in the energy function $U$ in order to force one of the coercivity conditions. In particular, we explore explicit constructions of measures on type-2 Carnot groups, which secure the Poincar{\' e} inequality and even the Log-Sobolev inequality. 
In Section 2-4, we consider additive and multiplicative ways of introducing singularity (by an additive term and a multiplicative factor given by positive singular homogeneous functions of horizontal coordinates). It is interesting that by our technique, we can provide new examples of probability measures for which Poincar{\'e} or even Log-Sobolev inequalities can be satisfied. Moreover, we are able to provide classes of examples going far beyond the existing results of 
\cite{HZ}, \cite{I}, \cite{BDZ1},  \cite{BDZ2}, \cite{ChFZ}.
In particular, we have results for all type-2 Carnot groups as well some others general cases (possibly satisfying certain technical conditions for homogeneous norms). \\
Some of our constructions could be interpreted as a perturbation of the harmonic oscillator by a singular potential, and therefore, they can be considered as generalisations of the Dyson-Ornstein-Uhlebeck models in Euclidean space (see e.g. \cite{D}, \cite{BCL} and references therein).\\

In Section 5, we explore a universality hypothesis, which suggests a classification framework for possible coercivity theory on Carnot groups.

In section 6, we discuss briefly some consequences of our results, including ergodicity in $\mathbb{L}_2$ and discreteness of spectrum of related Markov generators.

	\newpage
 \section{Additive Taming}
 \label{Additive Taming}
 \subsection{Poincar{\'e} and Log-Sobolev inequalities on  Carnot Groups of Type 2}$\;$\\ 
 \label{Sec.2}	
	
	\vspace{0.25cm} 
\noindent Let $\mathbb{G}$ be a step-2 group, i.e. a group isomorphic to ${   \mathbb{R}^{n+m}}$
	with the group law
	
	\[
	\left(w,z\right)\circ\left(w',z'\right)=\left(w_{i}+w'_{i},~z_{j}+z'_{j}+\frac{1}{2}<\Lambda^{\left(j\right)}w,w'>\right)_{i=1,..,n; j=1,..,m}
	\]
	for $w,w'\in\mathbb{R}^{n},z,z'\in\mathbb{R}^{m}$, where the matrices
	$\Lambda^{\left(j\right)}$  are $n \times n$ skew-symmetric and linearly independent.

	For $i\in\left\{ 1,\ldots,n\right\}$ and $j\in\left\{ 1,\ldots,m\right\} $, let 
	\[ X_{i}=\frac{\partial}{\partial x_{i}}+\frac{1}{2}\sum_{k=1}^{m}\sum_{l=1}^{n}\Lambda_{il}^{\left(k\right)}x_{l}\frac{\partial}{\partial z_{k}}\qquad \textrm{ and } \qquad Z_{j}=\frac{\partial}{\partial z_{j}}.
	\]
	Later on, $\nabla \equiv (X_i)_{i=1,..,n}$ and $\Delta \equiv \sum_{i=1,..,n} X_i^2 $ will denote the associated sub-gradient  
	and sub-Laplacian, respectively.
	We consider the following smooth homogeneous norm on $\mathbb{G}$
	
	\begin{equation}\label{2.N} N\equiv \left( |x|^4+ a|z|^2\right)^{\frac{1}{4}} ,
	\end{equation}
	with $a\in(0,\infty)$

Define 
\[ \label{4D.U.1}
U\equiv  V(\beta N+\xi(x)) ,
\]
with a real twice differentiable function $V$ and  $\xi(x)\equiv \xi(|x|)$ to be specified later, and some constant $\beta\in(0,\infty)$.
We will assume that 
\[ 0< Z\equiv \int e^{-U}d\lambda <\infty, \]
and consider the following probability measure on $\mathbb{G}$
\[d\mu_U \equiv \frac1{Z}e^{-U}d\lambda. \]

For $\sigma\in(0,\infty)$, consider
\begin{equation} 	\label{eq:4.D.xi.1}
	\xi\equiv \xi(|x|)\equiv\frac{1}{|x|^\sigma}, 
\end{equation}
Then,  
\begin{equation} \label{eq:4.D.xi.2}
	\nabla\xi(x)   = -\frac{\sigma}{|x|^{1+\sigma}}\cdot\frac{x}{|x|},\qquad
	\Delta\xi(x)    = -  \sigma\frac{n-2-\sigma}{|x|^{2+\sigma}} 
\end{equation}  
We have
\begin{equation} \label{4D.U.2}
	\begin{split}
		|\nabla U|^2   &= |\beta \nabla N + \nabla \xi|^2 (V')^2 \\ 
		\Delta U &= 
		(\beta \Delta N +\Delta \xi)V' +|\beta \nabla N + \nabla \xi|^2 V'',
	\end{split}
\end{equation}  
and hence,
\begin{equation} \label{4D.U.3}
	\begin{split}
		\mathcal{V}_2 \equiv	\frac14|\nabla U|^2 -\frac12 	\Delta U =
		|\beta\nabla N + \nabla \xi|^2 	\left(\frac14(V')^2 -\frac12 V''\right)\\ 
		-\frac12
		(\beta\Delta N +\Delta \xi)V'.
	\end{split}
\end{equation}  
From our choice of $\xi$ with \eqref{eq:4.D.xi.2}, we get
\begin{equation} \label{4D.U.4}
	\begin{split}
		\mathcal{V}_2 
		=
		\left|\beta\nabla N   -\frac{\sigma}{|x|^{1+\sigma}}\cdot\frac{x}{|x|}\right|^2 	\left(\frac14(V')^2 - \frac12 V''\right)\\ 
		+\frac12
		\left(\sigma\frac{n-2-\sigma}{|x|^{2+\sigma}} - \beta \Delta N \right)V'.
	\end{split}
\end{equation}  

\subsection{Heisenberg Group Case}\label{2.Case.1H}
\begin{theorem} $\;$\\
	Let $N$ be the Kaplan norm on the Heisenberg group and let	$
		\xi\equiv \xi(|x|)\equiv\frac{1}{|x|^\sigma}$,
		 with $ \sigma\in(0,\infty)$.
	Suppose $\exists R\in(0,\infty)$ $\forall s\geq R$ we have
	\begin{equation}
		\begin{split}
			&V'(s)s^{-7} \geq D  \\
		& sV'(s) \geq BV(s)	\\
			&(V'(s))^2/V''(s)\underset{s\to\infty}{\longrightarrow} \infty ,
		\end{split}
	\end{equation}
with some constants $B,D >0$.
Then the probability measure 
\[d\mu_U\equiv \frac1Z e^{- V(\beta N+\xi(x)) } d\lambda\] 
satisfies Logarithmic Sobolev Inequality 
\[
\mu_U \left(f^2\log\frac{f^2}{\mu_U f^2}\right)\leq c\; \mu_U |\nabla f|^2
\]
with a constant $c\in(0,\infty)$ independent of a function $f$.

\end{theorem}
\textit{Proof} :
For the Heisenberg group, with orthogonal matrices $\Lambda^{(k)}$ and the constant $a=16$, we have $|\nabla N| =  \frac{|x|}{N}$ and 
$\frac{x}{|x|}\cdot\nabla N=\frac{|x|^3}{N^3}$.
Since then $N$ is the Kaplan norm we get
$\Delta N= \frac{Q-1}{N}|\nabla N|^2$.
  Hence,
\[ \begin{split}
	\left|-\frac{\sigma}{|x|^{1+\sigma}}\cdot\frac{x}{|x|}+ \beta\nabla N\right|^{2}= \frac{\sigma^2}{|x|^{2(1+\sigma)}}  + 
	\beta^2  \frac{|x|^2}{N^2}
	- \frac{2\sigma\beta}{|x|^{1+\sigma}}\frac{|x|^3}{N^3}\\
	= \left(\frac{\sigma }{|x|^{(1+\sigma)}} 
	-  \beta \frac{|x|^3}{N^3}\right)^2
	+ 
	\beta^2  \frac{|x|^2}{N^2}\left(1
	-  \frac{|x|^4}{N^4} \right) \\
	= \left(\frac{\sigma }{|x|^{(1+\sigma)}} 
	-  \beta \frac{|x|^3}{N^3}\right)^2
	+ 
	\beta^2 a  \frac{|x|^2 z^2}{N^6}  
\end{split}
\]	
The first term vanishes only if
\[
N^4 
=\left( \frac{ \beta}{\sigma}\right)^{4/3} |x|^{\frac{4}{3}(4+\sigma) } 
\] 
i.e. we have
\[
az^2
=\left( \frac{\beta}{\sigma}\right)^{4/3} |x|^{\frac{4}{3}(4+\sigma)} - |x|^4
\]
but the solution $|x|=0, z=0$ has to be discarded.
The only other solution is given by
\[
|x|^{\frac{4}{3}(4+\sigma) -4}= \left( \frac{\sigma}{\beta}\right)^{4/3}, z=0.
\]
Thus $| \nabla \xi + \nabla N|^2$ does not vanish outside a compact ball.
For $ (\frac{\sigma }{1+\beta })^{\frac1{(1+\sigma)}}\equiv r_0 \geq |x|$
as well as for $|x|\geq r_0 $  with $|z|^2\geq \frac{1}{a\beta^2r_0^2}$ , we can now see that
\[	|\beta \nabla N + \nabla \xi|^2 \geq \frac{1}{N^6}.\] 
Thus assuming $\exists R\in(0,\infty)$ and $\exists B,D\in(0,\infty)$ such that $\forall s\geq R$ we have
\begin{equation}
	\begin{split}
		&V'(s)s^{-7} \geq D  \\
	&sV'(s) \geq BV(s)	\\
		&(V'(s))^2/V''(s)\underset{s\to\infty}{\longrightarrow} \infty ,
	\end{split}
\end{equation}
then there exists a ball and a constant $C\in(0,\infty)$ such that outside this ball we have
\[
|\nabla U|^2, U\leq C(\mathcal{V}_2+1).
\]
Hence by arguments of \cite{HZ} the corresponding measure satisfies Log-Sobolev inequality. 
\qed
\bigskip



\subsection{Type 2 Carnot Groups}\label{2.Case.2}  
In this section, we prove the following result.

\begin{theorem} $\;$\\
	Let \[ N\equiv \left( |x|^4+ a|z|^2\right)^{\frac{1}{4}} .\]
	with $a\in(0,\infty)$ be a smooth homogeneous norm on a type 2 Carnot Group.
	Let	$
	\xi\equiv \xi(|x|)\equiv\frac{1}{|x|^\sigma}$,
	with $ \sigma\in(0,n-2)$.
	Suppose $\exists R\in(0,\infty)$ $\forall s\geq R$ we have
	\begin{equation}
		\begin{split}
				&V'(s)s^{-7} \geq D  \\
			&sV'(s) \geq BV(s)	\\
			&(V'(s))^2/V''(s)\underset{s\to\infty}{\longrightarrow} \infty ,
		\end{split}
	\end{equation}
	Then if $a>0$ is sufficiently large, then the probability measure $\mu_U$ satisfies Logarithmic Sobolev Inequality 
	\[
	\mu_U \left(f^2\log\frac{f^2}{\mu_U f^2}\right)\leq c\; \mu_U |\nabla f|^2
	\]
	with a constant $c\in(0,\infty)$ independent of a function $f$.
	For any $a>$ the measure $\mu_U$ satisfies Poincar{\'e} inequality provided that
	\[ 
	\frac{1}{N^{2(1+\sigma)}} \left(\frac14(V')^2 - \frac12 V''\right)\underset{N\to\infty}{\longrightarrow}\infty .
	\]
	
\end{theorem}
The following lemma concerning certain analytic properties of the homogeneous norm $N$ was proven in \cite{BDZ1}. 

\begin{lemma} \label{Lem.2.1N} 
	There exist constants $A,C\in(0,\infty)$
	\begin{equation}\label{S.2 grad N}
			A\frac{|x|^2}{N^2}\leq\left|\nabla N\right|^{2}\leq C\frac{|x|^2}{N^2} .
	\end{equation}
	
	and 
	there exists a constant $B\in(0,\infty)$ such that
	\begin{equation}\label{S.2laplacian N}
			|\Delta N |\leq B \frac{|x|^2}{N^3}. 
	\end{equation}
	and
	
	\begin{equation}\label{S.2n.grad N}
			\frac{x}{|x|}\cdot \nabla N  = \frac{|x|^3}{N^3} 
	\end{equation}
	and
	\[  
	%
|X_{i}X_{j}N | 
		\leq \frac{C}{N} (1+|\nabla N|^2 )
	\]	
\end{lemma}
Note that, in general, using just considerations of homogeneity, since $N$ is homogeneous of degree $1,$ then, 
$|\nabla N|$ is homogeneous of degree $0 $, and one can only get a crude bound from above of the following  form 	 
\begin{equation} \label{eq:2.2}
	|\nabla N|\leq C.
\end{equation}
We note that, if \eqref{S.2 grad N} 
of Lemma 1 
is satisfied by choosing $a$ in the definition of the norm $N$ sufficiently large, we get $A\geq 1$. 
Hence,
\[ \begin{split}
	|\beta \nabla N + \nabla \xi|^2 &= \\	
	\left|-\frac{\sigma}{|x|^{1+\sigma}}\cdot\frac{x}{|x|}+ \beta\nabla N\right|^{2}&= \frac{\sigma^2}{|x|^{2(1+\sigma)}}  + 
	\beta^2 |\nabla N |^{2}
	- \frac{2\sigma\beta}{|x|^{1+\sigma}}\frac{|x|^3}{N^3}\\
	&\geq  \frac{\sigma^2}{|x|^{2(1+\sigma)}}  + 
	A\beta^2 \frac{|x|^2}{N^2}
	- \frac{2\sigma\beta}{|x|^{1+\sigma}}\frac{|x|^3}{N^3}\\
	&= \left(\frac{\sigma }{|x|^{(1+\sigma)}} 
	-  \beta \frac{|x|^3}{N^3}\right)^2
	+ 
	\beta^2  \frac{|x|^2}{N^2}\left(A
	-  \frac{|x|^4}{N^4} \right) .
\end{split}
\]	
If $A\geq 1$, similar arguments as in the case of Heisenberg group apply, and the corresponding measure $\mu_U$ satisfies Log-Sobolev inequality.\\

On the other hand, if $0<A < 1$, for $|x|\geq R$ , so that
$A\beta > \frac{2\sigma }{R^{1+\sigma}}$,
we have
\begin{equation}  \label{4D.c2.0}
	|\beta \nabla N + \nabla \xi|^2 \geq	
	\frac{\sigma^2}{|x|^{2(1+\sigma)}}  + 
	A\beta^2 \frac{|x|^2}{N^2}
	-\frac{2\sigma\beta}{|x|^{1+\sigma}}\frac{|x|^3}{N^3}\geq 	\frac{\sigma^2}{N^{2(1+\sigma)}} 
\end{equation}	  

At the same time, using 
\eqref{eq:4.D.xi.2} and Lemma \ref{Lem.2.1N}, for $|x|\leq R$, we have
\begin{equation} \label{4D.c2.1}
	\begin{split}
		\sigma\frac{n-2-\sigma}{|x|^{2+\sigma}} - \beta B \frac{|x|^2}{N^3}   \geq
		\sigma\frac{n-2-\sigma}{R^{2+\sigma}} - \beta B \frac{R^2}{N^3} > \frac{\sigma}2 
		\frac{n-2-\sigma}{N^{2+\sigma}},
	\end{split}
\end{equation}  
provided that $N\geq L$, where $L\in(0,\infty)$ satisfies
\begin{equation} \label{4D.c2.2}
N^3\geq 	L^3	 \geq 2\beta B \frac{R^{4+\sigma}}{	\sigma(n-2-\sigma)}  . 
\end{equation} 
Hencem we get that
\begin{equation} \label{4D.U.c2.3}
	\begin{split}
		\mathcal{V}_2 
		\geq
		\left|\beta\nabla N   -\frac{\sigma}{|x|^{1+\sigma}}\cdot\frac{x}{|x|}\right|^2 	\left(\frac14(V')^2 - \frac12 V''\right)\\ 
		+\frac12
		\left(\sigma\frac{n-2-\sigma}{|x|^{2+\sigma}} - \beta B \frac{|x|^2}{N^3} \right)V'
	\end{split}
\end{equation}  
goes to infinity in all direction with $N\to \infty$, provided that
\[ 
\frac{1}{N^{2(1+\sigma)}} \left(\frac14(V')^2 - \frac12 V''\right)\underset{N\to\infty}{\longrightarrow}\infty
\]
By arguments of \cite{HZ}, this implies the Poincar\'e inequality.  \\
\bigskip
\qed

\newpage
\subsection{Case 3:  Carnot Groups}\label{2.Case.3} 
Recall that if $N$ is the Kaplan Norm, we have 
\begin{equation}\label{KN}
	\Delta N=\frac{Q-1}{N}  |\nabla N|^2.
	\end{equation}

Then,
\begin{equation} \label{4D.U.5}
	\begin{split}
		\mathcal{V}_2 
		=
		|\beta \nabla N + \nabla \xi|^2 	\left(\frac14(V')^2 -\frac12 V''\right)\\ 
		-\frac12
		\left(\frac{(Q-1)\beta |\nabla N|^{2}}{N} +\Delta \xi\right)V'.
	\end{split}
\end{equation} 
Let
\begin{equation} 	\label{eq:4.D.c3xi.1}
	\xi(s)\equiv\log(\frac{1}{|x|}),  
\end{equation}
then  
\begin{equation} \label{eq:4.D.c3xi.2}
	\nabla\xi(x)   = -\frac{1}{|x|}\cdot\frac{x}{|x|},\qquad
	\Delta\xi(x)    = -  \frac{n-2}{|x|^2} .
\end{equation}  
\label{EndXi4.D}
Suppose $N$ is the Kaplan norm and consider
\[ U\equiv V(\beta N+\xi).\]
Using \eqref{KN} for Kaplan norm, we have 
\begin{equation} \label{4D.c.3.1}
	\begin{split}
		\mathcal{V}_2 
		=
		\left|\beta\nabla N   -\frac{1}{|x|}\cdot\frac{x}{|x|}\right|^2 	\left(\frac14(V')^2 - \frac12 V''\right)\\ 
		+\frac12
		\left( \frac{n-2}{|x|^2} - \beta \frac{(Q-1) |\nabla N|^{2}}{N} \right)V'
	\end{split}
\end{equation} 

First, we observe that, since $\nabla \xi= -\frac{1}{|x|}\cdot\frac{x}{|x|}$, so
\[
|\beta \nabla N + \nabla \xi|^2 \geq \left(\beta |\nabla N| -\frac{1}{|x|}
\right)^2.
\]
We note that,  if for any $\varepsilon\in(\frac12,2)$
\[
  |\nabla N| \leq    \frac{\varepsilon}{|x|} , 
\]
then we have
\[  \frac{n-2}{|x|^2} - \beta \frac{(Q-1) |\nabla N|^{2}}{N}  \geq  \frac{1}{|x|^2} \left( (n-2) -2\beta\frac{Q-1}{N}\right).  
\]
One can prevent the right-hand side to be less than zero for sufficiently large $N$. In particular, for $|\nabla N| \leq    \frac{\varepsilon}{|x|}  $  and  $N> 4\beta \frac{Q-1}{n-2} $, we have
\[  \frac{n-2}{|x|^2} - \beta \frac{(Q-1) |\nabla N|^{2}}{N}  \geq  \frac{1}{|x|^2} \left( \frac{n-1}{2}\right) \geq  \frac{1}{N^2} \left( \frac{n-1}{2}\right).
\]
In this case, by our assumption on $V(s),$ we get
\[
\frac12
\left( \frac{n-2}{|x|^2} - \beta \frac{(Q-1) |\nabla N|^{2}}{N} \right)V'(\xi N)  \geq
   \frac{n-1}{2N^2} V'(\xi N) \underset{N\to\infty}{\longrightarrow} \infty.
\]
\label{sprawdzic-niepotrzebne}
Since by homogeneity $\sup|\nabla N|\equiv C<\infty$, we have
\begin{equation} \label{4D.c.3.2}
	|\beta \nabla N + \nabla \xi|^2 \geq	
	\frac{1}{|x|^2}  
	-\frac{2\beta}{|x|}C .
\end{equation}	  
Hence, for $|x|\leq R$ , so that
$\frac{1}{1+2\beta C } >  R $,
we have
\begin{equation} \label{4D.c.3.3}
	|\beta \nabla N + \nabla \xi|^2 \geq
	\frac{1}{|x|} \left(		\frac{1}{|x|}  
	- 2 \beta C 
	\right) \geq \frac{1}{N}.
\end{equation}
On the other hand, for $|x|\geq R$, we have
\[
  \frac{n-2}{|x|^2} - \beta \frac{(Q-1) |\nabla N|^{2}}{N}   \geq
 \frac{n-2}{|x|^2} - \beta \frac{(Q-1) C^{2}}{N}   \geq \frac{n-2}{R^2} - \beta \frac{(Q-1) C^{2}}{N} /
\]
Thus, choosing 
\[ N \geq 2\beta \frac{(Q-1) C^{2}}{(n-2)}R^2, \]
we have in this region
\[
\frac{n-2}{|x|^2} - \beta \frac{(Q-1) |\nabla N|^{2}}{N}      \geq \frac{n-2}{2 R^2} . 
\]
Summarising, 
choosing $V$ so that
\[ \frac1s^2 V'(s) \underset{s\to\infty}{\longrightarrow} \infty,\]
we obtain
\[
\mathcal{V}_2 \underset{N\to\infty}{\longrightarrow} \infty.
\]
Hence, by arguments of \cite{HZ}, our measure $\mu_{V(N+\xi)}$ satisfies the Poincar{\'e} inequality. 
\label{EndSec.2}

\newpage
\section{Multiplicative Taming}
\label{Multiplicative Taming}

\subsection{Setup}	
Define 
\[
U\equiv  V( \xi(|x|)N),
\]
with a smooth norm $N$, real twice differentiable outside zero functions $V$ and $\xi\equiv \xi(|x|)\geq 1$ and 
 $\xi$ decreasing on $\mathbb{R}^+$. 
Note that in our case, the following holds
\begin{equation}	\label{eq:3.1.1}
	\begin{split}\nabla\xi(|x|) & = \frac{x}{|x|}\xi'\\
		\\
		\Delta\xi(|x|) & =\nabla\cdot\left(\frac{x}{|x|}\xi'\right)= 
		\frac{n-1}{|x|}\xi' + \xi''.
	\end{split}
\end{equation}  
We will assume that 
\[ 0< Z\equiv \int e^{-U}d\lambda <\infty, \]
and consider the following probability measure on $\mathbb{G}$
\[d\mu \equiv \frac1{Z}e^{-U}d\lambda. \]
In the current case, we have
\begin{equation} \label{eq:3.1.2}
	\begin{split}
		\mathcal{V}_2 
		&=
		| \nabla \log N + \nabla \log\xi|^2 \; (\xi N)^2	\left(\frac14(V')^2 -\frac12 V''\right)  \\
		&\\
		&-\frac12
		\left(  \frac{\Delta N}{N}  + 2    \xi' \frac{x}{|x|}  \cdot\frac {\nabla N}{\xi N} +  \frac{n-1}{|x|}\frac{\xi'}{\xi} +\frac{\xi''}{\xi}  \right) \xi N\;V'.
	\end{split}
\end{equation} 


\subsection{Type 2 Case} \label{3.2:Type 2 Case}

In this section, we prove the following result.
\begin{theorem}
	Let $G$ be  a Carnot group  of type $2$ with the horizontal dimension $n$. For $a\in(0,\infty)$, let \[N=(|x|^4+a|z|^2)^\frac14\] 
	be a smooth homogeneous norm on $G$ for which we have 
	\[ A\frac{|x|^2}{N^2} \leq |\nabla N|^2.\]
	Suppose $\xi(s)=s^{-\sigma}$ with $\sigma\in(0,\infty) $. \\
	The following measure 
	\[
	d\mu\equiv \frac1Z e^{-V(\xi N)} d\lambda
	\]
	satisfies the Logarithmic Sobolev inequality provided that one of the following conditions holds:\\
	\vspace{0.25cm}
	
(i)	If $A\geq 1$ and $\sigma\neq 1 $
for
	\[  V'(s) /s^2 \underset{s\to\infty}{\longrightarrow} \infty\]
	and if $\sigma=1$ and $A > 1$
		\[  V'(s) /s^2 \underset{s\to\infty}{\longrightarrow} \infty\]
	 
		\vspace{0.25cm}
		
	(ii)	If $A \leq 1$ and $n-2-\sigma\geq 2 $ 
	
	\[  V'(s) /s^2 \underset{s\to\infty}{\longrightarrow} \infty. \]
\end{theorem}

\textit{Proof}:
In case of type 2 Carnot groups with the smooth norm $N=(|x|^4+a|z|^2)^\frac14$, using Lemma \ref{Lem.2.1N},  
we have
\begin{equation} \label{eq:3.2.1}
	\begin{split}
		|  \nabla \log N +  \nabla \log \xi|^2  = \left(\frac{\xi'}{\xi}\right)^2  + 2\frac{\xi'}{\xi} \frac{x\cdot \nabla N}{|x| N}
		+ \left(\frac{\nabla N}{ N}\right)^2 \\
		\geq \left(\frac{\xi'}{\xi}\right)^2 +2\frac{\xi'}{\xi} \frac{|x|^3}{N^4} +A \frac{|x|^2}{N^4}
	\end{split}
\end{equation} 
with the constant $A$ given in Lemma \ref{Lem.2.1N}.
In particular for $\xi(s)=s^{-\sigma}$, we get

\begin{equation} \label{eq:3.2.2}
	\begin{split}
	|\nabla \log N +  \nabla \log \xi |^2  
	=  \frac{\sigma^2}{|x|^2}  - 2 \frac{\sigma}{|x|} \frac{|x|^3}{N^4} 
	+\frac{|\nabla N|^2}{N^2}     
		%
	\geq  \frac{1 }{|x|^2 }\left( \sigma    -   \frac{|x|^4}{N^4}\right)^2 +\frac{|x|^2}{N^4}(A - \frac{|x|^4}{N^4})
	\end{split}
\end{equation} 

If $A>1$, by taking $a\in(0,\infty)$ sufficiently large,  
we have 
\[
|\nabla \log N +  \nabla \log \xi |^2  
\geq \frac{1 }{|x|^2} \left(  \sigma    -   \frac{|x|^4}{N^4} \right)^2 +\frac{|x|^4}{N^6}(A^2 - 1).
\]	
Then, for any $R\in[0,\infty)$ and all $|x|\leq R$, if $N$ is such that $\frac{|x|^4}{N^4}\leq \frac1{2}\sigma$, we have
\[
|\nabla \log N +  \nabla \log \xi |^2  
\geq   \frac{\sigma^2}{4 R^2} . 
\]	
On the other hand, for all $|x|\geq R$, we have
\begin{equation} \label{eq:3.2.3}
|\nabla \log N +  \nabla \log\xi|^2\geq \frac{1}{N^2} \left(  \sigma    - 1  \right)^2 +\frac{R^2}{N^4}(A  - 1).
\end{equation}  
In this case, we also have
\begin{equation} \label{eq:3.2.4}
		\begin{split}
- \frac12\left(  \frac{\Delta N}{N}  + 2    \xi' \frac{x}{|x|}  \cdot\frac{\nabla N}{\xi N} +  \frac{n-1}{|x|}\frac{\xi'}{\xi} +\frac{\xi''}{\xi}  \right) &=
- \frac12\left(  \frac{\Delta N}{N}  -      \frac{2\sigma}{|x|}  \cdot \frac{|x|^3}{N^4}  -\sigma\frac{(n-2-\sigma)}{|x|^2}  \right)\\
&\geq  \frac{\sigma}2\frac{(n-2-\sigma)}{|x|^2} + \left(  \sigma    -\frac{B}2\right) \frac{|x|^2}{N^4},
	\end{split}
\end{equation}  
which diverges to $\infty$ if $|x|\to 0$, provided $\sigma<n-2$, and otherwise remains bounded.\\

Then, choosing 
\begin{equation} \label{eq:3.2.5}
\left(	\frac{(\sigma-1)^2}{N^2} +\frac{R^2(A -1)}{N^4}\right) V'(\xi N)\underset{N\to\infty}{\longrightarrow}\infty,
\end{equation} 
we secure the conditions for 
the Logarithmic Sobolev
inequality. 

\vspace{0.25cm}
\label{2nd Case: A less than 1}

\noindent If 
$0<A \leq 1$ 
and 
$n-2-\sigma \geq 0$, 
choosing 
$\sigma\geq 2$,
we have
\begin{equation} \label{eq:3.2.6}
	\begin{split}
		|\nabla \log N +  \nabla \log \xi |^2  
	&	\geq  \frac{\sigma^2}{|x|^2}  - 2 \frac{\sigma}{|x|} \frac{|x|^3}{N^4} 
		+\frac{|\nabla N|^2}{N^2}     
		= \frac{1 }{|x|^2 }\left( \sigma    -   \frac{|x|^4}{N^4}\right)^2 +\frac{|x|^2}{N^4}(A - \frac{|x|^4}{N^4})\\
		&\geq 
		\frac{1 }{|x|^2 }\left( \sigma    -   1\right)^2 -\frac{1}{N^2}(1-A ) \geq \frac{A}{N^2}.
	\end{split}
\end{equation} 
In this case, we also have
\begin{equation} \label{eq:3.2.7}
	\begin{split}
		- \frac12\left(  \frac{\Delta N }{N} + 2    \xi' \frac{x}{|x|}  \cdot\frac{\nabla N}{\xi N} +  \frac{n-1}{|x|}\frac{\xi'}{\xi} +\frac{\xi''}{\xi}  \right) &=
		- \frac12\left(  \frac{\Delta N }{N}  -      2\sigma  \cdot \frac{|x|^2}{N^4}  -\sigma\frac{(n-2-\sigma)}{|x|^2}  \right)\\
		&\geq \frac{\sigma}2\frac{(n-2-\sigma)}{|x|^2} + \left(\sigma    - \frac{B}2\right) \frac{|x|^2}{N^4},
	\end{split}
\end{equation}  
which diverges to $\infty$ if $|x|\to 0$ and otherwise remains bounded.\\

Then, choosing 
\begin{equation} \label{eq:3.2.8}
\frac{A^2}{N^2}   V'(\xi N)
\underset{N\to\infty}{\longrightarrow}\infty,
\end{equation} 
we secure the conditions for 
the Logarithmic Sobolev inequality. 
\qed

\newpage 

\subsection{Case: Kaplan Norm on a Carnot Group} \label{Type 2 Case: Kaplan Norm}
In this section, we prove the following general result.
\begin{theorem}
	Let $N$ be the Kaplan norm of a Carnot group $G$ of homogeneous dimension  $Q$ and the horizontal dimension $n$. Let $C\equiv \sup |\nabla N|$.
	Suppose $\xi(s)=s^{-\sigma}$ with $\sigma\leq 1$.
	Suppose  $C < \sigma$ and $ n-3\geq 2C$,  and that
	\[ (V'(s))^2/V''_+(s) \underset{s\to\infty}{\longrightarrow} \infty\]
	and for some $\varepsilon\in(0,\infty)$
	\[
	\varepsilon (V'(s))^2\geq \max (V(s),sV'(s)).
	\]
	Then, the probability measure
	\[
	d\mu\equiv \frac1Z e^{-V(\xi N)} d\lambda
	\]
	satisfies the Logarithmic Sobolev inequality.
\end{theorem}

Remark that, in particular, $V(s)=s^p$ with $p\geq 2$
satisfies the conditions of the above theorem.\
\vspace{0.25cm}

\textit{Proof}:
First, we note that in the special case of $N$ being the Kaplan norm, we have
\begin{equation} \label{eq:3.3.1}
	\begin{split}
			\mathcal{V}_2 
		&=
	\left| \nabla \log N +  \frac{\xi'}{\xi} \frac{x}{|x|}\right|^2 (\xi N )^2	\left(\frac14(V')^2 -\frac12 V''\right)\\
	&\qquad -\frac12
	\left(  \frac{(Q-1)  |\nabla N|^{2}}{N^2}  + 2    \xi' \frac{x}{|x|}  \cdot\frac{\nabla N }{\xi N}+  \frac{n-1}{|x|}\frac{\xi'}{\xi} +\frac{\xi''}{\xi}  \right) \xi N\;V' 	
	\end{split}
\end{equation} 
For $\xi(s)=s^{-\sigma}$,  with $C\equiv \sup|\nabla N|$, we have
\begin{equation} \label{eq:3.3.2}
	\begin{split}
		|  \nabla \log N +  \nabla \log \xi|^2  = \left(\frac{\xi'}{\xi}\right)^2  + 2\frac{\xi'}{\xi} \frac{x\cdot \nabla N}{|x| N}
		+ \left(\frac{\nabla N}{ N}\right)^2 \\
		= \left(\frac{\sigma}{|x|}\right)^2  - 2\frac{\sigma}{|x|} \frac{x\cdot \nabla N}{|x| N}
		+ \left(\frac{\nabla N}{ N}\right)^2\\
	\geq \left(\frac{\sigma}{|x|}\right)^2  - 2\frac{C \sigma}{|x|N}   
	+ \left(\frac{\nabla N}{ N}\right)^2	.
	\end{split}
\end{equation} 
Hence, for any $R\in(0,\infty)$ and all $|x|\leq R$ and $N\geq 4R C/\sigma,$ we have
\begin{equation} \label{eq:3.3.3}
		|  \nabla \log N +  \nabla \log \xi|^2   
		\geq  \frac{\sigma^2}{2R^2}   .
\end{equation} 
Without further information about $|\nabla N|$ it is difficult to get a useful lower bound. 
 
However, we notice that in general, if   we have
 $\sigma   > C  \equiv \sup |\nabla N|$, using $N \geq |x|$, we get

\begin{equation} \label{eq:3.3.4}
	\begin{split}
		|\nabla \log N +  \nabla \log \xi |^2  
		\geq   
	 \left(\frac{\sigma}{|x|}\right)^2  - 2\frac{\sigma}{|x|} \frac{x\cdot \nabla N}{|x| N}
	+ \left(\frac{\nabla N}{ N}\right)^2 \geq \left(\frac{\sigma}{|x|} - \frac{ |\nabla N|}{N}	 \right)^2 \\
	 \geq 
	\left(\frac{\sigma}{|x|} - \frac{C}{N}	 \right)^2 \geq  \frac{(\sigma - C)^2}{N^2}  	 .
	\end{split}
\end{equation}  
For the same $\xi(s)=s^{-\sigma}$, 
we have 
\begin{equation} \label{eq:3.3.5}
	\begin{split}
-\left(  \frac{(Q-1)  |\nabla N|^{2}}{N^2}  + 2    \xi' \frac{x}{|x|}  \cdot\frac{\nabla N}{\xi N} +  \frac{n-1}{|x|}\frac{\xi'}{\xi} +\frac{\xi''}{\xi} \right)\qquad\qquad\qquad \\
\qquad\qquad=
\sigma\frac{n-2-\sigma }{|x|^2} +2\sigma \frac{1}{|x|} \frac{x}{|x|}\cdot \frac{\nabla N }{N}
-  \frac{(Q-1)  |\nabla N|^{2}}{N^2} \\
\geq \sigma\frac{n-2-\sigma }{|x|^2} - 2\sigma \frac{C}{|x| N}  
-  \frac{(Q-1) C^{2}}{N^2} .
	\end{split}
\end{equation} 
Thus, the corresponding term in $\mathcal{V}_2$ is stable for small $|x|$.
 
In this way, we obtain a coercive bound with potential $\mathcal{V}_2,$ and using arguments of \cite{HZ}, we  arrive at the Logarithmic Sobolev inequality.

\qed

	\label{EndSec.3}
\newpage

\label{Sec.4A}
	\section{Multiplicative Taming II}
	\label{Multiplicative Taming 2}
 
 Let $\mathbb{\xi}(s)=\log\left(e+{   \frac{1}{s}}\right),$ and for $0<L<1,$ define
\[
\tilde{\xi}(|x|)=\begin{cases}
\xi(|x|) & if\;\;|x|<1\\
\frac{(|x|-L)^{2}}{(1-L)^{2}}\xi(|x|) & if\;\;|x|\geq1
\end{cases},
\]
Let 
\[
U_{\xi,V}\equiv (1+\tilde{\xi}(|x|))V(N),
\]
and consider the following measure
\[
{ d\mu_{\xi,V}=\frac{e^{- U_{\xi,V} }}{Z}d\lambda }
\]
defined with a homogeneous norm $N$
\begin{theorem}
Suppose $N$ is a Kaplan norm satisfying 
\[ x\cdot\nabla N\geq0,\]
and such that 
we have
\[
 V(N)/V'(N),\; V(N)/V''(N)\to 0 \qquad  as \; N\to \infty 
\]
Then, the measure $\mu_{\xi,V}$
satisfies the Logarithmic Sobolev inequality.  
\end{theorem}
\textbf{Remark:}  As we show in Appendix 1, the condition $x\cdot\nabla N\geq0$ holds for example in the class of generalised Heisenberg group.

\begin{proof}
\noindent To show that the measure $   \mu_{U_{\xi,V}}$ 
satisfies the Logarithmic Sobolev inequality, it is sufficient to show that under our conditions for some $\alpha \in(0,1)$, there exists a constant $C\in(0,\infty)$ such that for all $N$ sufficiently large, we have
\[
{U_{\xi,V}}\leq C\left((1-\alpha)|\nabla {U_{\xi,V}}|^{2}-\Delta {U_{\xi,V}}   +1\right).
\]
\\
First, we note two basic facts about   $\mathbb{\xi}(s)=\log\left(e+{   \frac{1}{s}}\right)$.\\
 Namely, we have
 \[  
 \xi'(s)  = 
-\frac{1}{(1+es)s} ,\qquad 
 \xi''(s) =\frac{1+2es }{(1+es)^{2}s^{2}}.
\]
For the case $|x|\geq1 $, we have
\[
\nabla\tilde{\xi}(|x|)={   \frac{2(|x|-L)}{(1-L)^{2}}\xi\frac{x}{|x|}-\frac{(|x|-L)^{2}}{(1-L)^{2}}\frac{x}{(1+e|x|)|x|^{2}},}
\]
and so
\[
\left|\nabla\tilde{\xi}(|x|)\right|^{2}={   \frac{4(|x|-L)^{2}}{(1-L)^{4}}\xi^{2}+\frac{(|x|-L)^{4}}{(1-L)^{4}}\frac{1}{(1+e|x|)^{2}|x|^{2}}-4\frac{(|x|-L)^{3}}{(1-L)^{4}}\frac{\xi}{(1+e|x|)|x|}.}
\]
and 
\[\begin{split}
\Delta \tilde{\xi}(|x|)=\frac{2\xi}{(1-L)^{2}}-\frac{4(|x|-L)}{(1-L)^{2}}\frac{1}{(1+e|x|)|x|}+\frac{2(|x|-L)}{(1-L)^{2}}\xi\frac{(n-1)}{|x|} \\
-\frac{(|x|-L)^{2}}{(1-L)^{2}}\left(\frac{(n-2)+(n-3)e|x|}{(1+e|x|)^{2}|x|^{2}}\right).
\end{split}
\]

First, consider the case $|x| \geq 1$.
For the Kaplan norm $N$, using the relation ${   \Delta N=\frac{(Q-1)|\nabla N|^{2}}{N}}$, we have

\begin{equation}
\begin{array}{ll}
  & {(1-\alpha)|\nabla {U_{\xi,V}}|^{2}-\Delta {U_{\xi,V}}}  = \\
& = {   (1-\alpha)\left(|\nabla\tilde{\xi}|^{2}V(N)^{2}+(1+\tilde{\xi})^{2}V'(N)^{2}|\nabla N|^{2}+2(1+\tilde{\xi})V(N)V'(N)\nabla\tilde{\xi}\cdot\nabla N\right)} 
\\
 & {   -\Delta\tilde{\xi}V(N)-2V'(N)\nabla\tilde{\xi}\cdot\nabla N-(1+\tilde{\xi})V''(N)|\nabla N|^{2}-(1+\tilde{\xi})V'(N)\frac{(Q-1)|\nabla N|^{2}}{N}},
\end{array}
\end{equation}
and hence
\begin{equation}
\begin{array}{ll}
 & {(1-\alpha)|\nabla {U_{\xi,V}}|^{2}-\Delta {U_{\xi,V}}}  = 
\\
 & ={   (1-\alpha)V(N)^{2}\left(
 \frac{4(|x|-L)^{2}}{(1-L)^{4}}\xi^{2}+\frac{(|x|-L)^{4}}{(1-L)^{4}}\frac{1}{(1+e|x|)^{2}|x|^{2}}
-4\frac{(|x|-L)^{3}}{(1-L)^{4}}\frac{\xi}{(1+e|x|)|x|}\right)}

 \\
\\
 & +V(N)\left({   
 -\frac{2\xi}{(1-L)^{2}} 
 +\frac{4(|x|-L)}{(1-L)^{2}}\frac{1}{(1+e|x|)|x|}
 -\frac{2(|x|-L)}{(1-L)^{2}}\xi\frac{(n-1)}{|x|} 
 +\frac{(|x|-L)^{2}}{(1-L)^{2}}\left(\frac{(n-2)+(n-3)e|x|}{(1+e|x|)^{2}|x|^{2}}\right)}\right)\\
\\
 & +2V'(N)\left((1-\alpha)(1+\tilde{\xi})V(N)-1\right)
 \left({ \frac{2(|x|-L)}{(1-L)^{2}}\xi-\frac{(|x|-L)^{2}}{(1-L)^{2}}\frac{1}{(1+e|x|)|x|}}\right){ \frac{x}{|x|}\cdot\nabla N}\\
\\
 & +(1+\tilde{\xi})|\nabla N|^{2}\left((1-\alpha)(1+\tilde{\xi})V'(N)^{2}-V''(N)-{   V'(N)\frac{(Q-1)}{N}}\right).
\end{array}\label{eq:1}
\end{equation}

If   ${   \frac{x}{|x|}\cdot\nabla N\geq0}$,
then the terms $-4\frac{(|x|-L)^{3}}{(1-L)^{4}}\frac{\xi}{(1+e|x|)|x|},-\frac{2\xi}{(1-L)^{2}},-\frac{2(|x|-L)}{(1-L)^{2}}\xi\frac{(n-1)}{|x|},  -\frac{(|x|-L)^{2}}{(1-L)^{2}}\frac{1}{(1+e|x|)|x|}$ in (\ref{eq:1}) are negative for  $|x|\geq1,$ and the next goal
is to bound them by the term $\frac{4(|x|-L)^{2}}{(1-L)^{4}}\xi^{2}$ in (\ref{eq:1}).  First, we show that
\[
4\frac{(|x|-L)^{3}}{(1-L)^{4}}\frac{\xi}{(1+e|x|)|x|}  {\leq} \frac{1}{1+e}\frac{4(|x|-L)^{2}}{(1-L)^{4}}\xi^{2},
\]
i.e.
\[
\frac{(|x|-L)}{(1+e|x|)|x|}  {\leq}\frac{1}{1+e} \log\left(e+{   \frac{1}{|x|}}\right),
\]
which is true since $0<L $ and $|x|\geq1.$ Secondly, we have
\[
\frac{2\xi V(N)}{(1-L)^{2}}  {\leq}\frac14(1-\alpha)V(N)^{2}\frac{4(|x|-L)^{2}}{(1-L)^{4}}\xi^{2}
\]
i.e.
\[
1  {\leq} \frac12(1-\alpha)V(N)\frac{(|x|-L)^{2}}{(1-L)^{2}}\log\left(e+{   \frac{1}{|x|}}\right),
\]
which is true for $N$ sufficiently large as $V(N)$ is unbounded increasing. Next, we prove that  for $N$ sufficiently large,
\[
\frac{2(|x|-L)}{(1-L)^{2}}\xi\frac{(n-1)}{|x|}V(N)  {\leq}\frac14 (1-\alpha)V(N)^{2}\frac{4(|x|-L)^{2}}{(1-L)^{4}}\xi^{2},
\]
i.e.
\[
\frac{(n-1)}{2(1-\alpha)}  {\leq}V(N)\frac{|x|(|x|-L)}{(1-L)^{2}}\log\left(e+{   \frac{1}{|x|}}\right),
\]
which is true. Finally, using the fact that $|\nabla N|\leq C,$
we look at
\[
\frac{C(|x|-L)^{2}}{(1-L)^{2}}\frac{V'(N)\left((1-\alpha)(1+\tilde{\xi})V(N)\right)}{(1+e|x|)|x|}  {\leq}(1-\alpha)V(N)^{2}\frac{4(|x|-L)^{2}}{(1-L)^{4}}\xi^{2},
\]
i.e.
\[ 
C\frac{(|x|-L)^{2}}{ ( 1+e |x|)|x|}  {\leq}\frac{V(N)}{V'(N)}\log\left(e+{   \frac{1}{|x|}}\right).
\]
Since $(|x|-L)^{2}\leq|x|(|x|-L)\leq|x|(1+e|x|)$
and $\frac{V(N)}{V'(N)}\to 0$ as $N\to \infty$ , then the above inequality is easily achievable with a small multiplier on the right-hand side. \\

\textbf{Remark:} \textit{In the case
where ${   \frac{x}{|x|}\cdot\nabla N<0}$, we would
require the term $ \frac{2(|x|-L)}{(1-L)^{2}}\xi $ of (\ref{eq:1}) to be bounded by the
term $ \frac{4(|x|-L)^{2}}{(1-L)^{4}}\xi^{2}$ of (\ref{eq:1}): 
\[
\frac{4(|x|-L)}{(1-L)^{2}}\xi V'(N)\left((1-\alpha)(1+\tilde{\xi})V(N)\right)  {\leq}(1-\alpha)V(N)^{2}\frac{4(|x|-L)^{2}}{(1-L)^{4}}\xi^{2},
\]
i.e.
\[
V'(N)(|x|-L)  {\leq}V(N),
\]
we would need the condition $NV'(N)\leq V(N),$ which  generally does not hold for polynomial $V(N)$.}\\
\vspace{0.25cm}

 \label{leq 1}To study the case $|x|<1,$ we note that

\[
\nabla\tilde{\xi}=-\frac{x}{(1+e|x|)|x|^{2}},
\qquad so\quad
|\nabla\tilde{\xi}|^{2}=\frac{1}{(1+e|x|)^{2}|x|^{2}},
\]
and
\[
\Delta\tilde{\xi}=-\frac{(n-2)+(n-3)e|x|}{(1+e|x|)^{2}|x|^{2}}.
\]
Thus
\[
\begin{array}{ll}
& {   (1-\alpha)|\nabla {U_{\xi,V}}|^{2}-\Delta {U_{\xi,V}}} =\\
&{   (1-\alpha)\left(|\nabla\tilde{\xi}|^{2}V(N)^{2}+(1+\tilde{\xi})^{2}V'(N)^{2}|\nabla N|^{2}+2(1+\tilde{\xi})V(N)V'(N)\nabla\tilde{\xi}\cdot\nabla N\right)}\\
\\
 & {   -\Delta\tilde{\xi}V(N)-2V'(N)\nabla\tilde{\xi}\cdot\nabla N-(1+\tilde{\xi})V''(N)|\nabla N|^{2}-(1+\tilde{\xi})V'(N)\frac{(Q-1)|\nabla N|^{2}}{N}}\\
\\
 & ={   (1-\alpha)\left(\frac{V(N)^{2}}{(1+e|x|)^{2}|x|^{2}}\right)}+{   2\left(1-(1-\alpha)(1+\tilde{\xi})V(N)\right)V'(N)\frac{x}{(1+e|x|)|x|^{2}}\cdot\nabla N}\\
\\
 & {   +\frac{(n-2)+(n-3)e|x|}{(1+e|x|)^{2}|x|^{2}}V(N)+(1+\tilde{\xi})|\nabla N|^{2}\left((1-\alpha)(1+\tilde{\xi})V'(N)^{2}-V''(N)-V'(N)\frac{(Q-1)}{N}\right)}\\
\\
 & \geq c  {U_{\xi,V}}.
\end{array}
\]
with some positive constant $c$.

The last inequality is true since by the conditions on $V(N)/V'(N),\; V(N)/V''(N)\to 0 $ as
 $N\to \infty$, and the fact that

\[
2{   \left((1-\alpha)(1+\tilde{\xi})V(N)\right)V'(N)\frac{x}{(1+e|x|)|x|^{2}}\cdot\nabla N}  {\leq}{   (1-\alpha)\left(\frac{V(N)^{2}}{(1+e|x|)^{2}|x|^{2}}\right)},
\]

i.e., since $|\nabla N|\leq C,$ it suffices to show

\[
{   C\log\left(e+{   \frac{1}{|x|}}\right)(1+e|x|)|x|V'(N)}  {\leq}{   V(N)}.
\]

Since $|x|<1$, for sufficiently large $N$ we have
\[
{   C\log\left(e+{   \frac{1}{|x|}}\right)(1+e|x|)|x|V'(N)}\leq C(1+e|x|)^{2}V'(N){   \leq C(1+e)^{2}V'(N)\leq \frac12 V(N)}, 
\]
which ends the proof.
\end{proof}
\label{EndSec.4}

\newpage
 
\section{Universality Hypothesis on  Carnot Groups}$\;$\\ 
\label{Sec.5} 

It is natural to expect that if we have a probability measure on a Carnot group which has a density $e^{-U}$ with respect to the Haar measure dependent on a smooth homogeneous norm
and satisfies certain coercive inequalities, then a similar property holds for any other smooth homogeneous norm on the same group. 

On the bases of our experience, \cite{HZ}, \cite{I}, \cite{BDZ1}, \cite{BDZ2}, \cite{ChFZ},
one could conjecture that for any Carnot group  and any smooth homogeneous norm $N,$ there exists $U=U(N)$ such that the corresponding measure $\mu_U$ satisfies the Poincar\'e Inequality.

On the other hand, we know that stronger Log-Sobolev type inequalities cannot be satisfied for smooth norm case, \cite{HZ}, Theorem 6.3, and \cite{BDZ2}, Theorem 10. One may conjecture that in this case we have such inequalities satisfied provided that we restrict ourselves to probability measures $e^{-U}$ with $U=U(K),$ where $K$ is a non-smooth homogeneous norm, and once it holds for one such norm, it holds for any other non-smooth norm, possibly with some adjustment of the function $U$.

Below, we provide an illustration in the indicated direction.\\ 
Recall the following definition

\textbf{Definition}:\\
We call homogeneous norm on (the homogeneous Carnot group) $G$,  a continuous  function $K :G \to [0,\infty)$   such that:

(i) $K(\delta\lambda (x)) = \lambda K(x) \forall \lambda>0\; and\; x\in G$; \\
(ii) $K(x) > 0$ iff $x \neq 0$.\\

$K$ is called symmetric iff \\
(iii) $K(x^{-1}) = K(x)$ for every $x\in G$.
\\

Let $G$ be a stratified group and $x = (x^{(1)},...,x^{(r)})\in G$.
In this context, we have the following family of  homogeneous norms
\[
|x|_G := \left(\sum_{j=1}^r |x^{(j)}|^{\frac{\alpha}{j}}\right)^{\frac1{\alpha}}
\]
with $\alpha\geq 1$. In  particular, we have the following   homogeneous norm on $G$ which is smooth out of the origin
\[
|x|_G := \left(\sum_{j=1}^r |x^{(j)}|^{\frac{2r!}{j}}\right)^{\frac1{2r!}} .
\] 
At the other hand we have the following example of a non-smooth homogeneous norm 
\[
\rho(x) := \sum_{j=1}^N |x_j|^{\frac1{\sigma_j}}.
\]
We remark that given two homogeneous norms $N, \tilde N$
and a strictly positive (smooth) function, one can define a new homogeneous norm as follows \cite{Acz}
\[ \label{K}
K\equiv K(d,N) = N \zeta\left( \frac{\tilde N}{N}. \right)
\]
 Moreover, if the function $\zeta$ is concave, then the expression on the right-hand side is the perspective function in the sense of \cite{M} which is jointly concave. Then, if both norms $N, \tilde N$ satisfy the triangle inequality, so does $K$.
In particular, this is the case of a pair consisting of the control distance $d$ and the Kaplan norm on the Heisenberg group \cite{C}. Then, choosing $\zeta$ to be a root function, we get a new perfect distance function in the form of the geometric mean of $d$ and $N$.

The following property of equivalence  of the homogeneous norms is well established, see e.g. \cite{BLU}. Let $K$ be a homogeneous norm on $G$. Then there exists a constant $c > 0$ such that
\begin{equation} \label{eq.5.0} 
\forall x\in G\qquad \qquad |\delta_{c^{-1}}(x)|_G \leq K(x), N(x), d(x) \leq |\delta_{c}(x)|_G .
\end{equation}

\noindent Using a homogeneous norm $K,$ we introduce the following  probability
measure \[ d\mu_{U,K}= e^{-U(K)} d\lambda,\] 
with the normalization constant $Z\in(0,\infty)$ .\\

\begin{theorem} \label{Thm5.6}
Let $K_0$ be a homogeneous norm on a Carnot group $\mathbb{G}$.
Suppose a probability measure $\mu\equiv \mu_{U,K_0}$ 
satisfies a coercive bound
\[
\int |f|^q \mathcal{V}_q(K_0) d\mu \leq C\int |\nabla f|^q d\mu +D \int | f|^q d\mu
\]
with some constants $C,D\in(0,\infty)$ independent of $f$, and 
\[   \mathcal{V}_q(K_0) \geq \alpha_q |U'(K_0)|^q\]
for some $\alpha_q\in(0,\infty)$ and some $K_0$ large enough.\\
Let  $K$ be a different homogeneous norm on $\mathbb{G}$.
Suppose $U'$ satisfies the doubling property  $|U'(c^2t)|\leq  A_c |U'(t)|$ with a constant $A_c\in(0,\infty)$.
There exists $\varepsilon\in(0,\infty)$ such that
\[ \|\nabla (K-K_0)\|<\varepsilon \]
implies
\[
\int |f|^q \left( \tilde {\mathcal{V}}_q(K_0)  \right) d\mu_{U,K} \leq   2^{q-1}C\int |\nabla f|^q d\mu_{U,K} +  D \int | f|^q d\mu_{U,K} ,
\]
with
\[\tilde {\mathcal{V}}_q(K_0)  \geq  \left(\alpha_q -  \frac12\left(\frac{2A_c\varepsilon}{q}\right)^q  \right)|U'(K_0) |^q  \]
for large $K_0$.
\end{theorem}
 
\begin{proof} We consider the case $q=1$. For other $q\in(0,\infty)$ the arguments are similar.
Replacing  $f\geq 0$ by $f e^{(U(K_0)-U(K))}$, 
we get
\[
\int f \mathcal{V}_1(K_0)   \leq  C \int |\nabla f| d\mu_{U,K}  + C\int f   |\nabla(U(K_0)-U(K))|  d\mu_{U,K} + D \int f d\mu_{U,K}.
\]
Since with $K_s = sK_0+ (1-s)K\leq c^2K_0$, using the fact that by our assumption $|U'(c^2t)|\leq  A_c |U'(t)|$, with some $A_c\in(0,\infty)$ independent of $t$, we have
\[ |\nabla(U(K)-U(K))| \leq \int_0^1 ds |U'(K_s) | \nabla (K-K_0)| \leq 
A_c \varepsilon |U'(K_0) |
\]
Hence, 
\[
\int f\left( \mathcal{V}_1(K_0) - A_c \varepsilon |U'(K_0) |\right) d\mu_{U,K}  \leq  C \int |\nabla f| d\mu_{U,K}   +   D \int f d\mu_{U,K},
\]
from which the desired property follows.
\end{proof}
 
The merit of the above theorem is that the perturbation theory does not require second order derivatives of the homogeneous norm $K$.
Thus, if for  $K_0$ we have a coercive bound, so we do for $K,$ 
as long as it is a small perturbation of the original homogeneous norm.\\
 We recall \cite{HZ}
that by arguments involving Leibniz rule and integration by parts, for $d\mu\equiv \frac1Ze^{-U(d)}d\lambda$ with
\[
\mathcal{V}_1(d) \equiv U'(d) |\nabla d|^2-\Delta d,
\]
we have 
\[
\int f \mathcal{V}_1(d) d\mu \leq \int |\nabla f| d\mu . 
\]
For the control distance $d,$ we have $|\nabla d|=1$, so if one has a well-behaved Laplacian of the distance (e.g. locally unbounded only from below and with moderate growth dominated by $U'(d)$ in the large), we have a theory admitting coercive inequalities which allows for small perturbations in which we need only smallness  of the sub-gradient of the difference of the homogeneous norms.
For an example of this type, see e.g. the case of Heisenberg group
in \cite{HZ}.
\\

We remark that along the line indicated above,  one can develop naturally a perturbation theory for the theory discussed in sections 1-4. for measures in which we were  taming the singularities as well as  those considered in \cite{BDZ1}, \cite{BDZ2}, \cite{ChFZ}.\\
 Although our description was focused on Poincar\'e and Log-Sobolev inequalities, one can provide similar development
for other inequalities as e.g. in   \cite{BDZ1} or in \cite{RZ} and \cite{RZ2} including in particular necessary and sufficient condition for exponential decay to equilibrium in Orlicz spaces.
 
\subsection{Examples} \label{Sec.5.1 Examples}

\begin{example} \label{Example.5.1}
Let $K=(1-\alpha)d+\alpha N$, for any homogeneous norm $N$.
Then, using the fact that for a homogeneous norm $|\nabla N|\leq C$ for some $C\in(0,\infty)$, if $C\geq |\nabla d|\geq \kappa >0$ for some  $\kappa\in(0,\infty)$, we have
\[
|\nabla K|^2\geq  (1-\alpha )^2|\nabla d|^2+ 2(1-\alpha )\alpha  \nabla d\cdot\nabla N +\alpha^2 | \nabla N|^2 \geq
(1-\alpha) \left((1-\alpha)\kappa^2 - 2\alpha C^2 \right)
\]
which is positive for $\alpha >0$ sufficiently small. 
Moreover, we have
\[
|\nabla K - \nabla d|=  \alpha |\nabla N-\nabla d|\leq  2 \alpha C.
\]
Hence the assumptions of the perturbation Theorem \ref{Thm5.6} can be satisfied by taking $\alpha>0$ sufficiently small.\\

\end{example} \label{EndeExample.5.1}

\begin{example}  \label{Example.5.2} In this example, we discuss a case of a semi-norm obtained as a mixture of control and Kaplan norms.\\
For our purposes, we are interested in properties of the sub-gradient and the sub-Laplacian of the homogeneous norms.
If  $d$ is the control distance and  $N$ is the Kaplan
norm associated to a sub-gradient  $\nabla\equiv (X_1,..,X_n)$, then
the following relations are satisfied, \cite{BLU}
\begin{equation} \label{eq.5.1}
|\nabla d|=1 \qquad \qquad \qquad\qquad  \Delta N = \frac{Q-1}{N} 
|\nabla N|^{2}
\end{equation}
where $|\cdot|$ denotes Euclidean norm in $\mathbb{R}^n$ and $\Delta$ denotes the sub-Laplacian.
We define a new homogeneous norm by
\[
K\equiv K(d,N) = d \zeta\left( \frac{N}{d} \right).
\]
with a smooth non-negative function $ \zeta$. This includes more general means than just weighted arithmetic mean.
Using this definition and  the equivalence relation of norms \eqref{eq.5.0}, one gets:
\begin{lemma} \label{Lem.2}
\begin{equation} \label{eq.5.2a}
(\inf \zeta ) d\leq K \leq  (\sup \zeta) d
\end{equation}

\begin{equation} \label{eq.5.2b}
|\nabla K| = \left(( \zeta - \frac{N}{d}\zeta')^2 |\nabla d|^2 +2 (\zeta - \frac{N}{d}\zeta') \zeta'  \nabla d\cdot \nabla N +  (\zeta')^2 |\nabla N |^2\right)^{\frac12}
\end{equation}

\begin{equation} \label{eq.5.2c}
\begin{split}
\Delta K  = \left( \zeta - \frac{N}{d}\zeta' \right) \Delta d  +
 \zeta'' \frac{1}{d}\left(  \nabla N  -     \frac{N}{d}  \nabla d  \right)^2  +\zeta' \Delta N
\end{split}
\end{equation}
 
\end{lemma}
\hfill $\circ$

If $N$ is the Kaplan norm, then \eqref{eq.5.2c}
reads

\begin{equation} \label{eq.5.2d}
\Delta K  = \left( \zeta - \frac{N}{d}\zeta' \right) \Delta d  +
 \zeta'' \frac{1}{d}\left(  \nabla N  -     \frac{N}{d}  \nabla d  \right)^2  +\zeta' \frac{Q-1}{N} 
|\nabla N|^{2}
\end{equation}
and, since $\zeta$ is a smooth function on an interval 
$[c^{-2},c^2]$ and both  $|\nabla d|$ and $\nabla N$ are bounded, the leading term on the right-hand side is the one containing $\Delta d$.\\

Suppose  $K_s\equiv d \zeta_s(\frac{N}{d})$ is  a one parameter differentiable interpolation between $d$ and $K$ with bounded derivative $\frac{d}{ds}\zeta_s\equiv\dot \zeta_s$, then we have
\[\begin{split}
K_s &\leq d\max_s \|\zeta_s\|,\qquad  |\dot K_s| \leq 
d\max_s \|\dot\zeta_s\| \\
|\nabla K_s| &\leq  |\nabla d| \; \left( \max_s \|\zeta_s\| +  \frac{N}{d} \max_s \|\zeta_s'\|   \right)+   |\nabla N|   \max_s \|\zeta_s'\|  ,
\\
 |\nabla \dot K_s| & \leq 
 |\nabla d| \; \left(\max_s \|\dot\zeta_s\| + \frac{N}{d} \max_s \|\dot\zeta_s'\|  \right) +   |\nabla N|  \max_s \|\dot \zeta_s'\|  
\end{split}
\]

\[\begin{split}
 |\nabla(U(d)-U(K))| &= |\nabla(\int_0^1 U'(K_s) \dot K_s  ds)| \\
&\leq  \int_0^1 |U''(K_s) \dot K_s | \; |\nabla K_s| ds + \int_0^1 |U'(K_s)|\; |\nabla \dot K_s | ds 
\end{split}
\]
 Hence, the assumptions of the perturbation Theorem \ref{Thm5.6} can be satisfied by taking $\zeta$ sufficiently close to one.\\
\end{example} \label{EndExample.5.2}
%
 
In the following example, we  illustrate the above in the case of homogeneous norm which is created using the geometric mean.

\begin{example} \label{Example.5.3}

Consider $K=d^{1-\alpha}N^{\alpha},$ where $0<\alpha<1.$  
Then we have
\begin{equation} \label{eq.Ex5.E3.1}
 |\nabla K|^2 =  |(1-\alpha) \left(\frac{N}{d}\right)^\alpha \nabla d + \alpha \left(\frac{d}{N}\right)^{1-\alpha}\nabla N |^2
\end{equation}
If
\[
\frac1c d \leq  N \leq c d ,
\]
then 
\[  \frac1c   \leq  \frac{N }{d},  \frac{d }{N} \leq c   \]
and hence, for  $C\geq |\nabla d|\geq \kappa >0$ and $|\nabla N| \leq C$ , we get
\begin{equation} \label{eq.Ex5.E3.2}
 |\nabla K|^2 \geq   (1-\alpha)^2  \frac{ \kappa^2}{c^{2\alpha}}  -2 (1-\alpha) \alpha  c |\nabla d|\; | \nabla N | \geq
(1-\alpha)^2  \frac{ \kappa^2}{c^{2\alpha}}  -2 (1-\alpha) \alpha  c C^2,
\end{equation}
which can be made strictly positive for $\alpha>0$ sufficiently small.
We also have
\[\begin{split}
|\nabla K -\nabla d| \leq |(1-\alpha) \left(\frac{N}{d}\right)^\alpha -1| |\nabla d| + \alpha \left(\frac{d}{N}\right)^{1-\alpha}|\nabla N |\\
 \leq \left( |(1-\alpha) c^\alpha -1|  + \alpha c^{1-\alpha}\right)C,
\end{split}
\]
which can be made sufficiently small for $\alpha>0$ sufficiently small.
Hence the assumptions of the perturbation Theorem \ref{Thm5.6} can be satisfied by taking $\alpha>0$ sufficiently small .\\
\vspace{0.25cm}

For the sub-Laplacian, we have
\begin{equation} \label{eq.Ex5.2c}
\begin{split}
\Delta K  = \left( 1-\alpha\right) \zeta  \Delta d  -
\alpha\left( 1-\alpha\right) \zeta  d  \left(  \frac1N\nabla N  -     \frac{1}{d}  \nabla d  \right)^2  +\alpha\zeta \frac{d}{N}  \frac{Q-1}{N} |\nabla N|^2.
\end{split}
\end{equation}
Thus, for large $N,$ the possible singular behaviour is determined by 
the first term on the right-hand side.\\

\noindent Hence, in particular 
for the   Heisenberg group, we have the following conclusion (which follows from the corresponding $U$-bound via arguments of \cite{HZ}):
\begin{theorem}
If a probability measure 
\[d\mu =e^{-d^p} d\lambda\]
satisfies Log-Sobolev inequality for $p>p_0>2$, then so does the measure
\[d\nu =e^{-K^{p/\alpha}} d\lambda.\]
\end{theorem}
\end{example}
 
 \section{Ergodicity and discreteness of Spectrum}


Firstly, we recall that (See \cite{BG} and \cite{RZ})
coercive inequality of log$^\beta$-Sobolev type
is equivalent to  the following Orlicz-Sobolev inequality 
\begin{equation} \label{SO}
 \|(f-\mu f)^2\|_{\Phi,\mu} \leq C \|\nabla f\|^2_{\mu}, \tag{O-S}
\end{equation}

where on the left-hand side one has an Orlicz norm and a constant $C\in(0,\infty)$ independent of $f$.\\
Given the inequality (O-S), a theorem of Cipriani \cite{Ci} implies the following result.

\begin{theorem}
Suppose the probability measure $\mu_{\xi,N}$ satisfies conditions of Theorem I with $I\in\{1,..,5\}$ so that Logarithmic Sobolev inequality holds with Dirichlet form
\[
\mathcal{E}(f)\equiv \mu_{\xi,N} |\nabla f|^2,
\]
with the sub-gradient $\nabla$. Then, the corresponding Markov generator $L$ has a purely discrete spectrum.
\end{theorem}

In the situation of the theorem above, by a result of Rothaus (see e.g.\cite{GZ},\cite{ABC}), the following Poincare inequality holds
\[
\mu(f-\mu f)^2 \leq 2C \mu|\nabla f|^2,
\]
which is equivalent to the following exponential decay to equilibrium in $\mathbb{L}_2$
\[
 \| e^{tL}f-\mu f\|_2^2 \leq   e^{-2m t} \| f-\mu f\|_2^2
. \]

In conclusion, we mention that there exists a large class of nilpotent Lie groups to explore in related directions of coercive inequalities, spectrum, and ergodicity estimates, which are for the moment technically very difficult. Another interesting direction involves the analysis of systems with singular potentials. In \cite{Ze}, crystallographic groups associated to nilpotent Lie groups are constructed which, in turn, provide a way to introduce and study Dunkl-type operators.

  \newpage
  	
  \paragraph*{Appendix.1 : Generalised Heisenberg Group } $\;$\\
  \label{Appx.1.TypeGHeisenberg}

  \noindent For $1\leq j\leq n$, let $L_j \in\mathbb{R}\setminus\{0\}$. In  $\mathbb{R}^{2n+1}$, consider the algebraic group law given by
  \[(x, t) \ast (y, s)=(x+y,t+s+\sum_{j=1}^n(L_j(x_jy_{j+n}-y_jx_{j+n})
  ). \]
  We have
  \[
  X_j=\begin{cases}
  	\partial_{x_j}-L_j x_{j+n}\partial_t,\qquad j=1,..,n\\
  	\partial_{x_j} + L_{j-n} x_{j-n}\partial_t,\qquad j=n+1,..,2n .
  \end{cases}
  \]
  Then,
  \[
  [X_j,X_k]=\begin{cases}
  	2 L_j \partial_t,\qquad k=n+j\\
  	0,\qquad\qquad otherwise
  \end{cases}
  \]
  Kaplan Norm 
 
  \[ \label{KNGH}
   N=
  \left( \left( \sum_{j=1,..,n}  2|L_j|(x_j^2 +  x_{j+n}^2)   \right)^2 +16 z^2 \right)^\frac14
  \]

  Sub-gradient of Kaplan norm:
  \[
  X_j N = \begin{cases}
  	\left(  |L_j| x_j ( \sum_{k=1,..,n} 2|L_k|( x_k^2 +   x_{k+n}^2))  - 4L_j x_{j+n}z  \right)\frac1{ N^3} ,\quad j=1,..,n\\
  	\left(   |L_{j-n}| x_{j} ( \sum_{k=1,..,n} 2 |L_k|( x_k^2 +   x_{k+n}^2))  + 4 L_{j-n} x_{j-n}z  \right)\frac1{ N^3} ,\quad j=n+1,..,2n
  \end{cases}
  \]
  Using this we have,
  \[
  \begin{split}
  	|\nabla N|^2&=\sum_{j=1}^{2n} |X_j N|^2 =\\
  	&	\frac{1}{N^6} \sum_{j=1}^{n}  \left(  |L_j|\; x_j ( \sum_{k=1,..,n} 2|L_k|( x_k^2 +    x_{k+n}^2))  - 4L_j x_{j+n}z \right)^2     \\
  	& +\frac{1}{N^6}\sum_{j=n+1}^{2n} \left(   |L_{j-n}| \; x_{j} ( \sum_{k=1,..,n} 2|L_k|( x_k^2 +   x_{k+n}^2))  + 4 L_{j-n} x_{j-n}z  \right)^2 \\
  	& =\frac{  \sum_{k=1,..,n} |L_k|( x_k^2 +   x_{k+n}^2) }{N^2}
  \end{split}
  \]
  and 
  \[
  \mathbf{x}\cdot \nabla N =  \frac{2 ( \sum_{k=1,..,n} |L_k|( x_k^2 +   x_{k+n}^2))^2}{N^3} .
  \]
  Summarising, we have:
  \begin{lemma} \label{Lem.3}
  	For the Kaplan norm of Generalised Heisenberg Group we have 
\begin{equation} \label{Apx.1.Lem}
\begin{split}
\min_k|L_k| \; \frac{| \mathbf{x}|^2}{N^2} \leq 	|\nabla N|^2&=\frac{  \sum_{k=1,..,n} |L_k|( x_k^2 +   x_{k+n}^2) }{N^2} \leq \max_k|L_k| \; \frac{| \mathbf{x}|^2}{N^2}\\
	 \mathbf{x}\cdot \nabla N &=  \frac{2 ( \sum_{k=1,..,n} |L_k|
	 	( x_k^2 +   x_{k+n}^2))^2}{N^3}  \geq 0.
\end{split}
\end{equation}
  \end{lemma}

\paragraph*{Appendix.2 : Smooth Norms for Type 2  Nilpotent Lie Group } $\;$\\
\label{Appx.2.Type2G}

Let $\mathbb{G}$ be a step-2 group, i.e. a group isomorphic to ${   \mathbb{R}^{n+m}}$
with the group law

\[
\left(w,z\right)\circ\left(w',z'\right)=\left(w_{i}+w'_{i},~z_{j}+z'_{j}+\frac{1}{2}<\Lambda^{\left(j\right)}w,w'>\right)_{i=1,..,n; j=1,..,m}
\]
for $w,w'\in\mathbb{R}^{n},z,z'\in\mathbb{R}^{m}$, where the matrices
$\Lambda^{\left(j\right)}$  are $n \times n$ skew-symmetric and linearly independent.

For $i\in\left\{ 1,\ldots,n\right\}$ and $j\in\left\{ 1,\ldots,m\right\} $, let 
\[ X_{i}=\frac{\partial}{\partial x_{i}}+\frac{1}{2}\sum_{k=1}^{m}\sum_{l=1}^{n}\Lambda_{il}^{\left(k\right)}x_{l}\frac{\partial}{\partial z_{k}}\qquad \textrm{ and } \qquad Z_{j}=\frac{\partial}{\partial z_{j}}.
\]
Later on, $\nabla \equiv (X_i)_{i=1,..,n}$ and $\Delta \equiv \sum_{i=1,..,n} X_i^2 $ will denote the associated sub-gradient  
and sub-Laplacian, respectively.
We consider the following smooth homogeneous norm on $\mathbb{G}$

\[ N\equiv \left( |x|^4+ a|z|^2\right)^{\frac{1}{4}} \]
with $a\in(0,\infty)$.
\\
Another norm is
\[ \tilde N\equiv \left(  \left( |x|^4+ a|z|^2\right)^\frac12      + |x|^2    \right)^{\frac{1}{2}} \]
with $a\in(0,\infty)$.

\paragraph*{Appendix.3 : Sub-gradient and Sub-Laplacian of $K$} \label{Appx.3}$\;$\\
Define
\[
K\equiv K(d,N) = d \zeta\left( \frac{N}{d} \right).
\]
Due to the equivalence of homogeneous norms, the function $\zeta$ is a smooth function supported in a bounded interval $[-c^{-1},c]$.

We have
\begin{equation} \label{nabla K}
|\nabla K| = |(\zeta -  \frac{N}{d } \zeta') \nabla d +  (\zeta')\nabla N | = 
\left( (\zeta -   \frac{N}{d  } \zeta')^2 +2 \nabla d\cdot \nabla N(\zeta -   \frac{N}{d  } \zeta')(\zeta')   +  (\zeta')^2|\nabla N |^2 \right)^\frac12,
\end{equation}
and 
\begin{equation} \label{Delta K}
\begin{split}
\Delta K &=\nabla\cdot\left(( \zeta - \frac{N}{d}\zeta') \nabla d + \zeta' \nabla N\right)\\
&=\left( \zeta - \frac{N}{d}\zeta' \right) \Delta d \\
& +
\zeta'  \left( \frac{\nabla N}{d}- \frac{ N}{d}\frac{\nabla d}{d} + \frac{N}{d}\frac{\nabla d}{d} - \frac{\nabla N}{d}\right) \cdot \nabla d  
\\
&- \frac{N}{d}\zeta'' \left( \frac{\nabla N}{d} - \frac{N}{d}\frac{\nabla d}{d} \right)\cdot \nabla d+\\
&  +\zeta' \Delta N +
 \zeta'' \left(\frac{|\nabla N|^2}{d} - \frac{N}{d} \frac{\nabla N\cdot\nabla d}{d} \right)\\
&\\
&=\left( \zeta - \frac{N}{d}\zeta' \right) \Delta d \\
&\zeta'' \left(\frac{|\nabla N|^2}{d} - 2 \frac{N}{d} \frac{\nabla N\cdot \nabla d}{d} + \frac{N}{d} \frac{N}{d}\frac{\nabla d\cdot \nabla d}{d} \right)+\\
&  +\zeta' \Delta N \\
&\\
&=\left( \zeta - \frac{N}{d}\zeta' \right) \Delta d  + \zeta'' \frac{1}{d}\left(  \nabla N  -     \frac{N}{d}  \nabla d  \right)^2   +\zeta' \Delta N\\
& = \left( \zeta - \frac{N}{d}\zeta' \right) \Delta d
+\zeta'' \frac{1}{d}\left(  \nabla N  -     \frac{N}{d}  \nabla d  \right)^2  
+\zeta' \frac{Q-1}{N} 
|\nabla N|^{2}.
\end{split}
\end{equation}
Since $\zeta$ is a smooth function on a bounded interval $[-c^{-1},c]$, for some positive constant $c$, the leading term in the last formula  outside a large ball is provided by 
$\zeta\Delta d$.

\vspace{\baselineskip} 

\textbf{Data availability:} Data sharing is not applicable to this article, as no datasets were generated or analysed during the current study.
 	\label{bibliography}	
	
\label{EndBibliography}
 \end{document}